\documentclass[11pt,letterpaper]{amsart}
\usepackage{graphicx} 
\usepackage{amsmath}
\usepackage{amsfonts}
\usepackage{amssymb}
\usepackage{amsthm}
\usepackage{xcolor}

\newtheorem{lem}{Lemma}
\newtheorem{thm}{Theorem}
\newtheorem{prop}{Proposition}

\newcommand\blfootnote[1]{%
  \begingroup
  \renewcommand\thefootnote{}\footnote{#1}%
  \addtocounter{footnote}{-1}%
  \endgroup
}

\title[A formula for the Euler characteristic of a poset]{A formula for the Euler characteristic of a poset through the determinant of the order-complement matrix}
\author{Pedro J. Chocano and L. Felipe Prieto-Martínez}

\begin{document}

\maketitle

\begin{abstract}
Given a finite poset $P$, its zeta matrix $\mathbf Z$ encodes fundamental incidence-theoretic information about the order structure. In this paper we introduce and study the \emph{order-complement matrix} $\overline{\mathbf Z} = \mathbf J - \mathbf Z$, where $\mathbf J$ is the all-ones matrix. We prove a closed formula for its characteristic polynomial and for its determinant, showing that $\det(\overline{\mathbf Z}) = (-1)^{n+1} \tilde{\chi}(P)$,
where $n = |P|$ and $\tilde{\chi}(P)$ is the reduced Euler characteristic of $P$. This provides a new, unexpectedly simple linear-algebraic expression for the Euler characteristic of a poset, complementing existing determinant formulas for matrices derived from incidence relations.
\end{abstract}

\blfootnote{2020  Mathematics  Subject  Classification: 06A11, 06A07.}
\blfootnote{Keywords: reduced Euler characteristic, poset, determinant, characteristic polynomial.}
\blfootnote{This research is partially supported by Grant  PID2024-156663NB-I00 from Ministerio de Ciencia, Innovación y Universidades (Spain).}

\section{Introduction}

A \textbf{partially ordered set}, or \textbf{poset}, is a set $P$ equipped with a binary relation that is reflexive, antisymmetric, and transitive. This relation, usually denoted by $\leq$, provides a way to compare certain pairs of elements within the set, though not necessarily all of them.

In a poset $(P,\leq)$, a \textbf{chain} is a subset of elements $C=\{x_0,x_1,\ldots, x_k\}\subseteq P$ 
in which every pair of elements is comparable and arranged in a strictly increasing order, that is, 
\[
x_0 < x_1 < \cdots < x_k.
\]
\noindent The \textbf{length} of a chain is defined as the number of strict comparisons, that is, in the previous case, the length of $C$ is $k$.

The \textbf{Euler characteristic} $\chi(P)$ of a finite poset $P$, $|P|=n$, is defined as the alternating sum of the number of chains of each length in $P$:
$$\chi(P)=c_0-c_1+c_2-\ldots+(-1)^{n-1}c_{n-1}.$$

\noindent Equivalently, it can be computed as the alternating sum of the ranks of the homology groups of its order complex.  This invariant captures topological information about the poset and is preserved under order-theoretic equivalences such as homotopy equivalence of order complexes. The \textbf{reduced Euler characteristic} is defined as
\begin{equation}
\label{eq.reduced}\tilde{\chi}(P) = \chi(P) - 1.
\end{equation}

Given a finite poset $(P,\leq)$, $P=\{x_1,\ldots, x_n\}$, its \textbf{zeta matrix} is the $n\times n$ square matrix $\mathbf Z$ whose entries are defined by
$$z_{ij}=\begin{cases} 1 &\text{if }x_i\leq x_j \\ 0&\text{otherwise.}\end{cases}$$

\noindent This matrix is invertible.

Let us introduce another matrix which, up to our knowledge, is not so extended in the literature. We define the \textbf{order-complement matrix} $\overline{\mathbf Z}$ of the poset as 
\begin{equation} \label{eq.comp} \overline{\mathbf Z}=\mathbf J-\mathbf Z, \end{equation}

\noindent where $\mathbf J$ is the $n\times n$ matrix whose entries are all equal to 1 and $\mathbf Z$  is the zeta matrix. The entries in such matrix are:
\begin{equation}\label{eq.zbarra}\overline{z}_{ij}=\begin{cases} 0 &\text{if }x_i\leq x_j \\ 1&\text{otherwise.}\end{cases}\end{equation}

In this paper, we prove the following result:

\begin{thm} \label{thm.main} Let $(P,\leq)$ be a poset and let $\mathbf Z, \overline{\mathbf Z}$ be, respectively, its zeta matrix and its order-complement matrix.

\begin{itemize}
    \item[(a)] The characteristic polynomial of $\overline{\mathbf Z}$ is
\begin{equation}\label{eq.poly} p(\lambda)=(-1)^n (\lambda+1)^n - \sum_{k=0}^{n-1} (-1)^{k+n} c_k \, (\lambda+1)^{\,n-1-k},\end{equation}

\noindent where 
$$c_k=\mathbf 1^T(\mathbf Z-\mathbf I)^k\mathbf 1$$

\noindent is the number of chains of length $k$ and $\mathbf 1$ is the column vector whose entries are all equal to 1.

\item[(b)] Given a finite poset $(P,\leq)$, let  $\overline{\mathbf Z}$ be its order-complement matrix. Then 
$$\det\overline{\mathbf Z}=(-1)^{n+1} \tilde{\chi}(P).$$

\end{itemize}

\end{thm}

\section{Some basic results}

In this section, we present the main ingredients for the proof of Theorem \ref{thm.main}. All of them are well-known in the literature but, for the sake of completeness, we include their proof.

The first result is a consequence of the comment after the definition of \emph{zeta matrix} in \cite{BallantineFrechetteLittle2004}:

\begin{lem}\label{lem.1}
Let $(P,\leq)$ be a finite poset with zeta matrix $\mathbf Z$. 

\begin{itemize}

    \item[(a)] The characteristic polynomial of $\mathbf Z$ is $(\lambda-1)^n$, 1 is the only eigenvalue of $\mathbf Z$, and $\det(\mathbf Z) = 1$.
    \item[(b)]  tr$(\overline{\mathbf Z})=0$.
\end{itemize}
\noindent

\end{lem}

\begin{proof} Any finite poset admits a \emph{linear extension}. A \textbf{linear extension} of a finite poset $(P,\leq)$ is a total order
$\preceq$ on the elements of $P$ that is compatible with the original
partial order, that is, whenever $x \leq y$ in the poset,
we also have $x \preceq y$.  Equivalently, a linear
extension is a listing $x_1,x_2,\dots,x_n$ of the elements of $P$ such
that $x_i \leq x_j$ in $P$ implies $i<j$. We refer the reader to \cite{S} for more information.

So we can assume that $\mathbf Z$ is upper unitriangular. From an upper triangular matrix, it is easy to obtain the characteristic polynomial, the eigenvalues and the determinant.

Finally, 
$$\text{tr}(\overline{\mathbf Z})=\text{tr}(\mathbf J)-\text{tr}(\mathbf Z)=n-n=0.$$
\end{proof}

Associated with a poset, we can also find the \textbf{Möbius matrix}, defined as the inverse of the zeta matrix. Its  entries form the Möbius function of the poset. For more information, we refer the reader to \cite{R,S}.

 Viewing $P$ as a category, the following result is the specialization of \emph{Leinster’s Formula} valid for any finite category with Möbius inversion (Section 2, after Definition~2.2 in \cite{Leinster2008}). Nevertheless, we will give our own proof of this result, based in the Neumann series expansion for the inverse:

\begin{lem}\label{lem.sum}
Let $P$ be a finite poset with Möbius matrix $\mathbf Z^{-1}$. Then the sum of all entries of the M\"obius matrix equals the Euler characteristic of $P$, that is,
\begin{equation} \label{eq.sum}
\chi(P)=\mathbf 1^T\mathbf Z^{-1}\mathbf 1,
\end{equation}

\noindent where $\mathbf 1$ is the column vector, whose entries are all equal to 1.
\end{lem}

\begin{proof}
We can assume, again, that
the zeta matrix $\mathbf Z$ is upper unitriangular.

Write $\mathbf Z=\mathbf I+\mathbf N$. Then $\mathbf N$ is strictly
upper triangular and therefore nilpotent: there exists $m\le |P|$ with $\mathbf N^m=\mathbf 0$. Since $\mathbf N$ is nilpotent, the inverse of $\mathbf I+\mathbf N$ is
given by the finite Neumann series:
\[
\mathbf Z^{-1}
= \sum_{k=0}^{m-1}(-\mathbf N)^k
= \mathbf I - \mathbf N + \mathbf N^2 - \cdots
  + (-1)^{m-1}\mathbf N^{m-1}.
\]

The entry in the position $(i,j)$ in $\mathbf N^k$ counts the number of strict chains of length $k$ with minimum $x_i$ and maximum $x_j$. Hence $\mathbf 1^T \mathbf N^k \mathbf 1$ is the total number
of strict chains in $P$ having $k$ inequalities, equivalently chains with
$k+1$ elements. Let $c_k$ denote this number. The Euler characteristic of $P$ is therefore
\[
\chi(P)=\mathbf 1^T\mathbf Z^{-1}\mathbf 1=\sum_{k\ge 0} (-1)^k c_k.
\]

\end{proof}

\section{Proof of Theorem \ref{thm.main}} \label{section.proof}

\subsection*{Proof of Statement (a)}

We can assume, again, that the zeta matrix $\mathbf Z$ is upper unitriangular.

Observe that $\mathbf J = \mathbf{1}\mathbf{1}^T$, where  $\mathbf{1}$ is the column vector of all ones. Let $\mathbf A := \lambda \mathbf I + \mathbf Z$.  For $\lambda\neq -1$, the matrix $\mathbf A$ is invertible and, in this case, from Equation \eqref{eq.comp} and the matrix determinant lemma, we obtain: 
\[
p(\lambda)=\det \big( \overline{\mathbf Z} -\lambda \mathbf{I})
= (-1)^n\det\big(\mathbf A - \mathbf{1}\mathbf{1}^T\big)
= (-1)^n \det\big(\mathbf A\big)\Big(1 - \mathbf{1}^T \mathbf A^{-1}\mathbf{1}\Big).
\]


\noindent From Lemma \ref{lem.1}, Statement (a), we know that $\det\big(\mathbf A\big)=(\lambda+1)^n.$ So 
\begin{equation} \label{eq.preneumann}
p(\lambda)
= (-1)^n(\lambda+1)^n\Big(1 - \mathbf{1}^T \mathbf A^{-1}\mathbf{1}\Big).
\end{equation}

Since $\mathbf{Z}$ is upper unitriangular, we can write $\mathbf{Z} = \mathbf{I} + \mathbf{N}$, where $\mathbf{N}$ is strictly upper triangular. Then
\[
\mathbf{A} = \lambda \mathbf{I} + \mathbf{Z} = (\lambda+1)\mathbf{I} + \mathbf{N}.
\]

\noindent For $\lambda\neq -1$, because $\mathbf{N}$ is nilpotent, we can use the Neumann series to compute the inverse:
\[
\mathbf{A}^{-1} = \frac{1}{\lambda+1} \left( \mathbf{I} + \frac{-\mathbf{N}}{\lambda+1} \right)^{-1} 
= \frac{1}{\lambda+1} \sum_{k=0}^{n-1} \left( -\frac{\mathbf{N}}{\lambda+1} \right)^k
= \sum_{k=0}^{n-1} \frac{(-1)^k \mathbf{N}^k}{(\lambda+1)^{k+1}}.
\]

Substituting the previous expression into Equation \eqref{eq.preneumann}, we prove Equation \eqref{eq.poly}, for $\lambda\neq-1$. Finally, since both sides of Equation \eqref{eq.poly} are polynomials and they coincide in $\mathbb R\setminus\{-1\}$, we conclude that the equality also holds for $\lambda=-1$.

\subsection*{Proof of Statement (b)}

The determinant of $\overline{\mathbf Z}$ can be obtained from the characteristic polynomial $p(\lambda)$ evaluated at $\lambda=0$, i.e., 
\[
\det(\overline{\mathbf Z}) = p(0).
\]

\noindent Substituting $\lambda = 0$ into $p(\lambda)$ in Equation \eqref{eq.preneumann}, we get
$$\det(\overline{\mathbf Z})
=(-1)^n\Big(1 - \mathbf{1}^T \mathbf Z^{-1}\mathbf{1}\Big)=(-1)^{n+1}\Big(\mathbf{1}^T \mathbf Z^{-1}\mathbf{1}-1\Big).$$

The result follows from the previous expression and Equations \eqref{eq.reduced} and \eqref{eq.sum}.

\section{Brief literature review}

To the best of our knowledge, no determinant formula involving the order-complement matrix has appeared before in the literature. Our result provides a new linear-algebraic characterization of the reduced Euler characteristic.

In the study of determinants of matrices naturally associated with posets, there is a classical and well-known work by Ballantine, Frechette, and Little, who consider the matrix  
\[
\mathfrak{Z} = \mathbf Z + \mathbf Z^T,
\]  
where \(\mathbf Z\) is the zeta matrix of a finite poset \(P\). They derive a combinatorial interpretation of the determinant of \(\mathfrak{Z}\) and provide a recursive formula in the case where \(P\) is a Boolean algebra \cite{BallantineFrechetteLittle2004}.


On a broader conceptual level, the notion of Euler characteristic has been extended well beyond posets and topological complexes to more abstract structures such as categories. In his influential paper \cite{Leinster2008}, Leinster defines the Euler characteristic of a finite category by generalizing Rota’s Möbius inversion from posets to categories; his definition works when the category admits both a weighting and a coweighting. He then proves that this new invariant behaves in a manner entirely compatible with classical invariants (e.g., under colimits), thus situating the Euler characteristic in a very general categorical framework. Beyond Leinster’s combinatorial definition, there are further developments exploring alternative approaches: for instance, Berger and Leinster reinterpret the Euler characteristic of a category as the “summation” of a (formally divergent) alternating series; under suitable interpretations, this divergent series can be assigned a finite value that agrees with Leinster’s characteristic~\cite{BergerLeinster2008}.

In the realm of posets but still in relation to the Euler characteristic, Noguchi studies the zeta series (a generating function related to the zeta function) under iterative barycentric subdivision~\cite{Noguchi2016}. While that work is focused on analytic properties, asymptotic behavior, and zeros of the zeta series rather than on determinants of matrices like \(\mathbf J - \mathbf Z\), it nonetheless shows the richness of connections between zeta functions, poset structure, and topological invariants.

It is worth noting that the first appearance of the so-called order-complement matrix in the literature—though not under this terminology—occurs in \cite{Chocano2025preprint}.  In that paper, the author introduced this matrix independently of the works cited above and showed, among other topological and combinatorial aspects, that its determinant and rank are homotopy and simple homotopy invariants. The author also raised the question of whether the absolute value of the determinant equals the absolute value of the reduced Euler characteristic, a question mentioned in the introduction following a referee’s observation. Later, in \cite{Chocano2026}, the same author analyzes eigenvalues of quadratic forms related to this matrix and their connections with the topological properties of posets.

\section{Future Work}

Let $(P,\leq)$ be a finite poset. The spectrum of the zeta matrix $\mathbf Z$ of the poset does not contain any information (see Lemma \ref{lem.1}, Statement (a)). But the determinant formula proved here suggests that the spectrum of the order-complement matrix $\overline{\mathbf Z}$ may encode additional invariants of the poset. In particular, it would be natural to investigate how the distribution of eigenvalues of $\overline{\mathbf Z}$ reflects combinatorial properties of $P$, or whether specific families of posets give rise to recognizable spectral patterns. Understanding these eigenvalues could lead to spectral characterizations of classical order-theoretic invariants and provide new tools for comparing posets via linear-algebraic data.

Just as a first step in this direction, let us include the following:

\begin{prop} Let $(P,\leq)$ be a finite poset and let $\overline{\mathbf Z}$ be its order-complement matrix. Then:
\begin{itemize}

    \item[(a)] The sum of the eigenvalues of $\overline{\mathbf Z}$ is always 0.
\item[(b)] Suppose that $(P,\leq)$ has a maximum (equiv. a minimum) $x$. Then 0 is an eigenvalue of $\overline{\mathbf Z}$. 

Moreover, if $\overline{\mathbf Z}'$ is the order-complement matrix corresponding to the poset $(P\setminus\{x\},\leq)$, then the spectra of $\mathbf Z,\overline{\mathbf Z}'$ are related by
$$\sigma(\overline{\mathbf Z}')=\sigma\{\overline{\mathbf Z}\}\cup\{0\}.$$

\item[(c)] If $(P,\leq)$ is a totally ordered set, then the only eigenvalue of $\overline{\mathbf Z}$ is 0.

\end{itemize}
    
\end{prop}

\begin{proof}

A) is a consequence of Lemma \ref{lem.1} (Statement (b)). In turn, Statement (b) follows from Equation \eqref{eq.zbarra} and the Laplace expansion for the determinant. Finally, Statement (c) easily follows from induction in the size of $P$ (using Statement (b)).
    
\end{proof}

\textbf{Acknowledgments.}
We are grateful to Professor Martin Rubey for carefully reading our manuscript and pointing out two errors. His observations allowed us to correct these mistakes and improve the quality of this work.

\end{document}